\documentclass[12pt]{article}
\usepackage{pslatex}
\usepackage{amsmath,amsthm,amssymb,amsfonts, stmaryrd}
\usepackage{fancybox}
\usepackage{color, graphicx}
\usepackage[usenames,dvipsnames,svgnames,table]{xcolor}
 \usepackage[T1]{fontenc}
 \usepackage{hyperref}
\newtheorem{theorem}{Theorem}[section]

\newtheorem{corollary}[theorem]{Corollary}

\newtheorem{proposition}[theorem]{Proposition}

\newtheorem{remark}[theorem]{Remark}

\newcommand{\Z}{\mathbb Z}

\newcommand{\N}{\mathbb N}
\newcommand{\Q}{\mathbb Q}

\begin{document}

\title{An adaptive upper bound on the Ramsey numbers $R(3,\dots,3)$ }

\author{S. Eliahou}

\date{}

\maketitle

\begin{abstract} Since 2002, the best known upper bound on the Ramsey numbers $R_n(3)=R(3,\dots,3)$ is $R_n(3) \le n!(e-1/6)+1$ for all $n \ge 4$. It is based on the current estimate $R_4(3) \le 62$. We show here how any closing-in on $R_4(3)$ yields an improved upper bound on $R_n(3)$ for all $n \ge 4$. For instance, with our present adaptive bound, the conjectured value $R_4(3)=51$ implies $R_n(3) \le n!(e-5/8)+1$ for all $n \ge 4$.
\end{abstract}

\section{Introduction}
For $n \ge 1$, the $n$-color Ramsey number $R_n(3)=R(3,\dots,3)$ denotes the smallest $N$ such that, for any $n$-coloring of the edges of the complete graph $K_N$, there is a monochromatic triangle. See e.g.~\cite{GRS, Ram, Soi} for background on Ramsey  theory. There is a well known recursive upper bound on $R_n(3)$ due to \cite{GG}, namely
\begin{equation}\label{rec}
R_n(3) \le n(R_{n-1}(3)-1)+2
\end{equation}
for all $n \ge 2$. Currently, the only exactly known values of $R_n(3)$ are $R_1(3)=3$, $R_2(3)=6$ and $R_3(3)=17$. As for $n=4$, the current state of knowledge is 
$$
51 \le R_4(3) \le 62.
$$
The lower bound is due to~\cite{Ch} and the upper bound to~\cite{FKR}, down from the preceding bound $R_4(3) \le 64$ in~\cite{S-F}. Moreover, it is conjectured in~\cite{XuR} that
$$
R_4(3) = 51.
$$
Here is a brief summary of successive upper bounds on $R_n(3)$. In \cite{GG}, the authors proved that
$$
R_n(3) \le n!e+1
$$
for all $n \ge 2$. Whitehead's results~\cite{Whi} led to
$$
R_n(3) \le n!(e-1/24)+1
$$
for all $n \ge 2$, and Wan~\cite{Wan} further improved it to
$$
R_n(3) \le n!(e-e^{-1}+3)/2+1.
$$
The last improvement came in 2002, when it was proved in \cite{XuXC} that
$$
R_n(3) \le n!(e-1/6)+1
$$
for all $n \ge 4$. That bound relies on the estimate $R_4(3) \le 62$ by \cite{FKR}. 

\smallskip
Because of the recurrence relation~\eqref{rec}, any improved upper bound on $R_{k}(3)$ for some $k \ge 4$ will yield an improved upper bound on $R_{n}(3)$ for all $n \ge k$. Our purpose here is to make this automatic improvement explicit. For instance, combined with our adaptive upper bound, the above-mentioned conjecture $R_4(3) = 51$ implies 
$$R_n(3) \le n!(e-5/8)+1$$
for all $n \ge 4$. This would be a substantial improvement over the current upper bound $n!(e-1/6)+1$, since $e-1/6 \approx 2.55$ while $e-5/8 \approx 2.09$.

\section{Main results}
As reported in~\cite{radzi}, it is proved in~\cite{XuXC} that 
$$
R_n(3) \le n!(e-1/6)+1
$$
for all $n \ge 4$. But the latter paper is in Chinese and not easily accessible to English readers. In this section, we prove a somewhat more general statement. We shall need the formulas below.

\medskip
\subsection{Useful formulas}
In proving $R_n(3) \le n!e+1$, the authors of ~\cite{GG} used without comment the formula
$$
\lfloor (n+1)!e \rfloor = (n+1)\lfloor n!e \rfloor+1
$$
for all $n\ge 1$. For convenience, we provide a proof here, as a direct consequence of the auxiliary formula below.

\begin{proposition}\label{n!e} For all $n \ge 1$, we have
$
\lfloor n!e \rfloor = \sum_{i=0}^n\, {n!}/{i!} \,.
$
\end{proposition}
\begin{proof} We have $ e = 1/0!+1/1!+\sum_{i=2}^{\infty} 1/i! = 2+\sum_{i=2}^{\infty} 1/i!$. Since $e < 3$, it follows that $\sum_{i=2}^{\infty} 1/i! < 1$. Now $n!e= \sum_{i=0}^n\, {n!}/{i!} + \sum_{i=n+1}^\infty\, {n!}/{i!}$. The left-hand summand is an integer, while the right-hand one satisfies 
$$
\sum_{i=n+1}^\infty\, {n!}/{i!} = \sum_{j=1}^\infty\, \frac{1}{\Pi_{k=1}^j (n+k)} \le \sum_{i=2}^\infty\, 1/{i!} < 1. 
$$ This concludes the proof.
\end{proof}
\begin{corollary}[\cite{GG}]\label{fact} For all $n \ge 1$, we have
$
\lfloor (n+1)!e \rfloor = (n+1)\lfloor n!e \rfloor+1.
$
\end{corollary}
\begin{proof} Applying Proposition~\ref{n!e} for $n+1$ and then for $n$, we have
\begin{eqnarray*}
\lfloor (n+1)!e \rfloor & = & \sum_{i=0}^{n+1}\, {(n+1)!}/{i!} \\
& = & (n+1)\sum_{i=0}^{n}\, {n!}/{i!} + (n+1)!/(n+1)! \\
& = & (n+1)\lfloor n!e \rfloor+1. \qedhere
\end{eqnarray*}
\end{proof}

\subsection{An optimal model}
We now exhibit an optimal model for the recursion~\eqref{rec}.

\begin{proposition}\label{lem rec} Given $q \in \Q$, let $f \colon \N \to \Z$ be defined by
$f(n)= \lfloor n!(e-q) \rfloor+1$ for $n \in \N$. Then, for all $n \in \N$ such that $n!q \in \Z$, we have
\begin{equation}\label{rec f}
f(n+1) = (n+1)(f(n)-1)+2.
\end{equation}
\end{proposition}

\begin{proof} 
We have
\begin{eqnarray*}
f(n+1) & = & \lfloor (n+1)!(e-q) \rfloor +1 \\
& = & \lfloor (n+1)!e \rfloor - (n+1)!q+1 \,\,\, \textrm{ [since $(n+1)!q \in \Z$]} \\
& = & (n+1)\lfloor n!e \rfloor +1 - (n+1)!q+1 \,\,\, \textrm{[by Corollary~\ref{fact}]} \\
& = & (n+1)\lfloor n!(e-q) \rfloor +2  \,\,\, \textrm{ [since $n!q \in \Z$]} \\
& = & (n+1)(f(n)-1)+2. \qedhere
\end{eqnarray*}
\end{proof}

\subsection{An adaptive bound}
Our adaptive upper bound on $R_n(3)$ is provided by the following statements.

\begin{proposition}\label{main prop} Let $k \in \N$ and $q \in \Q$ satisfy $k \ge 2$, $R_k(3) \le k!(e-q)+1$ and $k!q \in \N$. Then $R_n(3) \le n!(e-q)+1$ for all $n \ge k$.  
\end{proposition}
\begin{proof} As in Proposition~\ref{lem rec}, denote $f(n)=\lfloor n!(e-q) \rfloor+1$ for $n \in \N$. By assumption, we have
\begin{equation}\label{bound}
R_k(3) \le f(k)
\end{equation}
and $k!q \in \Z$. It suffices to prove the claim for $n=k+1$, since if $k!q \in \N$ then $(k+1)!q \in \N$. By successive application of \eqref{rec}, \eqref{bound} and \eqref{rec f}, we have
\begin{eqnarray*}
R_{k+1}(3) & \le & (k+1)(R_k(3)-1)+2 \\
& \le & (k+1)(f(k)-1)+2 \\
& = & f(k+1).
\end{eqnarray*}
Note that using \eqref{rec f} is allowed by Proposition~\ref{lem rec} and the assumption $k!q \in \N$.
\end{proof}

\begin{theorem}\label{main thm} Let $k \ge 2$ be an integer. Let $a \in \N$ satisfy $a \le \lfloor k!e \rfloor - R_k(3)+1$, and let $q=a/k!$. Then $R_n(3) \le n!(e-q)+1$ for all $n \ge k$.  
\end{theorem}
\begin{proof} We have $a \le k!e-R_k(3)+1$, so $R_k(3) \le k!e-a+1=k!(e-q)+1$. Moreover $k!q=a \in \N$. The conclusion follows from Proposition~\ref{main prop}.
\end{proof}

\begin{remark} Theorem~\ref{main thm} is the best possible application of Proposition~\ref{main prop}. Indeed, with the value $a'=\lfloor k!e \rfloor - R_k(3)+2$ and $q'=a'/k!$, it no longer holds that $R_k(3) \le k!(e-q')+1$.
\end{remark}

\subsection{The case $k=4$}
We now apply the above result to the case $k=4$. We only know $51 \le R_4(3) \le 62$ so far. Note that by Proposition~\ref{n!e}, we have 
\begin{equation}\label{4!e}
\lfloor 4!e \rfloor = \sum_{i=0}^4 4!/i! = 24+24+12+4+1= 65.
\end{equation}

\begin{proposition}\label{a} Let $a \in \N$ satisfy $a \le 66-R_4(3)$. Then setting $q=a/24$, we have $R_n(3) \le n!(e-q)+1$ for all $n \ge 4$.
\end{proposition}
\begin{proof} By \eqref{4!e}, $a$ satisfies the hypotheses of Theorem~\ref{main thm}. The conclusion follows.
\end{proof}
When the exact value of $R_4(3)$ will be known, Proposition~\ref{a} will provide an adapted upper bound on $R_n(3)$ for all $n \ge 4$. In the meantime, here are three possible outcomes.
\begin{corollary}[\cite{XuXC}]\label{62} $R_n(3) \le n!(e-1/6)+1$ for all $n \ge 4$.
\end{corollary}
\begin{proof} Since $R_4(3) \le 62$, we may take $a=4$ in Proposition~\ref{a}. The conclusion follows from that result with $q=a/4!=1/6$. 
\end{proof}

Note that the above bound dos not extend to $n=3$, since $R_3(3)=17$, whereas by Proposition~\ref{n!e}, we have $\lfloor 3!(e-1/6)\rfloor +1= \lfloor 3!e\rfloor = 3!+3!+3+1=16$.

\medskip
As mentioned earlier, it is conjectured in~\cite{XuR} that $R_4(3)=51$. If true,  Proposition~\ref{a} will yield the following improved upper bound.

\begin{corollary}\label{51} If $R_4(3)=51$, then $R_n(3) \le n!(e-5/8)+1$ for all $n \ge 4$.
\end{corollary}
\begin{proof} By Proposition~\ref{a}, with $a = 66-51=15$ and $q=15/4!=5/8$.
\end{proof}
As noted in the Introduction, this would be a substantial improvement over the current upper bound $n!(e-1/6)+1$, since $e-1/6 \approx 2.55$ whereas $e-5/8 \approx 2.09$. 

\smallskip
An intermediate step would be, for instance, to show $R_4(3) \le 54$ if at all true. This would yield the following weaker improvement.

\begin{corollary}\label{54} If $R_4(3) \le 54$, then $R_n(3) \le n!(e-1/2)+1$ for all $n \ge 4$.
\end{corollary}

\begin{proof} By Proposition~\ref{a}, with $a = 66-54=12$ and $q=a/4!=1/2$.
\end{proof}

\begin{remark} The above three corollaries are best possible applications of Proposition~\ref{a}, as in each case we took the largest admissible value for $a \in \N$.
\end{remark}

\subsection{The case $k=5$}
Let us also briefly consider the case $k=5$. At the time of writing, we only know $162 \le R_5(3) \le 307$. See~\cite{radzi}.

\begin{proposition}\label{b} Let $a \in \N$ satisfy $a \le 327-R_5(3)$. Then setting $q=a/120$, we have $R_n(3) \le n!(e-q)+1$ for all $n \ge 5$.
\end{proposition}
\begin{proof} By Theorem~\ref{main thm} and the value $\lfloor 5!e \rfloor=326$ given by Proposition~\ref{n!e}.
\end{proof}

Here again are three possible outcomes. Knowing only $R_5(3) \le 307$ does not allow to improve the current estimate $R_n(3) \le n!(e-1/6)+1$. At the other extreme, if $R_5(3) = 162$ holds true, it would yield $R_n(3) \le n!(e-11/8)+1$ for all $n \ge 5$. As an intermediate estimate, if $R_5(3) \le 227$ holds true, it would imply $R_n(3) \le n!(e-5/6)+1$ for all $n \ge 5$.

\section{Concluding remarks}

\subsection{On $\lim_{n \to \infty} R_n(3)^{1/n}$}

The adaptive upper bound on $R_n(3)$ given by Theorem~\ref{main thm} may still be quite far from reality, as the asymptotic behavior of $R_n(3)$ remains poorly understood. For instance, is there a constant $c$ such that $R_{n+1}(3) \le cR_n(3)$  for all $n$? Or, maybe, such that $R_n(3) \ge c n!$ for all $n$? The former would imply that $\lim_{n \to \infty} R_n(3)^{1/n}$, known by~\cite{ChG} to exist, is finite, whereas the latter would imply $\lim_{n \to \infty} R_n(3)^{1/n}=\infty$. At the time of writing, it is not known whether that limit is finite or infinite. See e.g.~\cite{LRX}, where this question is discussed together with related problems.

\subsection{Link with the Schur numbers}
The Schur number $S(n)$ is defined as the largest integer $N$ such that for any $n$-coloring of the integers $\{1,2,\dots,N\}$, there is a monochromatic triple of integers $1 \le x,y,z \le N$ such that $x+y=z$. The existence of $S(n)$ was established by Schur in~\cite{Schur}, an early manifestation of Ramsey theory. Still in~\cite{Schur}, Schur proved the upper bound
\begin{equation}\label{S(n)}
S(n) \le n!e-1
\end{equation}
for all $n \ge 2$. The similarity with the upper bound $R_n(3) \le n!e+1$ proved 40 years later in~\cite{GG} is striking. In fact, there is a well known relationship between these numbers, namely
\begin{equation}\label{S and R}
S(n) \le R_n(3)-2.
\end{equation}
Thus, via~\eqref{S and R}, our adaptive upper bound on $R_n(3)$ given by Theorem~\ref{main thm} also yields an upper bound on $S(n)$.

\bigskip
\noindent
\textbf{Acknowledgements.} The author wishes to thank L. Boza and S.P. Radziszowski for informal discussions during the preparation of this note and for their useful comments on a preliminary version of it.

{\small

\bigskip
\noindent
\textbf{Author's adress:} \\ Shalom Eliahou, Univ. Littoral C\^ote d'Opale, \\ EA 2597 - LMPA - Laboratoire de Math\'ematiques Pures et Appliqu\'ees Joseph Liouville, F-62228 Calais, France, \\
CNRS, FR 2956, France. \\ E-mail: \url{eliahou@univ-littoral.fr}
}

\end{document}